\documentclass[twoside,leqno,10pt, A4]{amsart}
\usepackage{amsfonts}
\usepackage{amsmath}
\usepackage{amscd}
\usepackage{amssymb}
\usepackage{amsthm}
\usepackage{amsrefs}
\usepackage{latexsym}
\usepackage{mathrsfs}
\usepackage{bbm}
\usepackage{amscd}
\usepackage{amssymb}
\usepackage{amsthm}
\usepackage{amsrefs}
\usepackage{latexsym}
\usepackage{mathrsfs}
\usepackage{bbm}
\usepackage{enumerate}
\usepackage{graphicx}
\usepackage{color}
\begin{document}

\newtheorem{theorem}[subsection]{Theorem}
\newtheorem{proposition}[subsection]{Proposition}
\newtheorem{lemma}[subsection]{Lemma}
\newtheorem{corollary}[subsection]{Corollary}
\newtheorem{conjecture}[subsection]{Conjecture}
\newtheorem{prop}[subsection]{Proposition}
\newtheorem{defin}[subsection]{Definition}

\numberwithin{equation}{section}
\newcommand{\mr}{\ensuremath{\mathbb R}}
\newcommand{\mc}{\ensuremath{\mathbb C}}
\newcommand{\dif}{\mathrm{d}}
\newcommand{\intz}{\mathbb{Z}}
\newcommand{\ratq}{\mathbb{Q}}
\newcommand{\natn}{\mathbb{N}}
\newcommand{\comc}{\mathbb{C}}
\newcommand{\rear}{\mathbb{R}}
\newcommand{\prip}{\mathbb{P}}
\newcommand{\uph}{\mathbb{H}}
\newcommand{\fief}{\mathbb{F}}
\newcommand{\majorarc}{\mathfrak{M}}
\newcommand{\minorarc}{\mathfrak{m}}
\newcommand{\sings}{\mathfrak{S}}
\newcommand{\fA}{\ensuremath{\mathfrak A}}
\newcommand{\mn}{\ensuremath{\mathbb N}}
\newcommand{\mq}{\ensuremath{\mathbb Q}}
\newcommand{\half}{\tfrac{1}{2}}
\newcommand{\f}{f\times \chi}
\newcommand{\summ}{\mathop{{\sum}^{\star}}}
\newcommand{\chiq}{\chi \bmod q}
\newcommand{\chidb}{\chi \bmod db}
\newcommand{\chid}{\chi \bmod d}
\newcommand{\sym}{\text{sym}^2}
\newcommand{\hhalf}{\tfrac{1}{2}}
\newcommand{\sumstar}{\sideset{}{^*}\sum}
\newcommand{\sumprime}{\sideset{}{'}\sum}
\newcommand{\sumprimeprime}{\sideset{}{''}\sum}
\newcommand{\sumflat}{\sideset{}{^\flat}\sum}
\newcommand{\shortmod}{\ensuremath{\negthickspace \negthickspace \negthickspace \pmod}}
\newcommand{\V}{V\left(\frac{nm}{q^2}\right)}
\newcommand{\sumi}{\mathop{{\sum}^{\dagger}}}
\newcommand{\mz}{\ensuremath{\mathbb Z}}
\newcommand{\leg}[2]{\left(\frac{#1}{#2}\right)}
\newcommand{\muK}{\mu_{\omega}}
\newcommand{\thalf}{\tfrac12}
\newcommand{\lp}{\left(}
\newcommand{\rp}{\right)}
\newcommand{\Lam}{\Lambda_{[i]}}
\newcommand{\lam}{\lambda}
\newcommand{\af}{\mathfrak{a}}
\newcommand{\sw}{S_{[i]}(X,Y;\Phi,\Psi)}
\newcommand{\lz}{\left(}
\newcommand{\pz}{\right)}
\newcommand{\bfrac}[2]{\lz\frac{#1}{#2}\pz}
\newcommand{\odd}{\mathrm{\ primary}}
\newcommand{\even}{\text{ even}}
\newcommand{\res}{\mathrm{Res}}
\newcommand{\sumn}{\sumstar_{(c,1+i)=1}  w\left( \frac {N(c)}X \right)}
\newcommand{\lab}{\left|}
\newcommand{\rab}{\right|}
\newcommand{\Go}{\Gamma_{o}}
\newcommand{\Ge}{\Gamma_{e}}
\newcommand{\M}{\widehat}

\theoremstyle{plain}
\newtheorem{conj}{Conjecture}
\newtheorem{remark}[subsection]{Remark}

\makeatletter
\def\widebreve{\mathpalette\wide@breve}
\def\wide@breve#1#2{\sbox\z@{$#1#2$}%
     \mathop{\vbox{\m@th\ialign{##\crcr
\kern0.08em\brevefill#1{0.8\wd\z@}\crcr\noalign{\nointerlineskip}%
                    $\hss#1#2\hss$\crcr}}}\limits}
\def\brevefill#1#2{$\m@th\sbox\tw@{$#1($}%
  \hss\resizebox{#2}{\wd\tw@}{\rotatebox[origin=c]{90}{\upshape(}}\hss$}
\makeatletter

\title[Ratios conjecture for quadratic twist of modular $L$-functions ]{Ratios conjecture for quadratic twist of modular $L$-functions }

\author[P. Gao]{Peng Gao}
\address{School of Mathematical Sciences, Beihang University, Beijing 100191, China}
\email{penggao@buaa.edu.cn}

\author[L. Zhao]{Liangyi Zhao}
\address{School of Mathematics and Statistics, University of New South Wales, Sydney NSW 2052, Australia}
\email{l.zhao@unsw.edu.au}

\begin{abstract}
 We develop the $L$-functions ratios conjecture with one shift in the numerator
and denominator in certain ranges for the family of quadratic twist of modular $L$-functions using multiple
Dirichlet series under the generalized Riemann
hypothesis.
\end{abstract}

\maketitle



\section{Introduction}\label{sec 1}

  As a significant conjecture with many important applications, the $L$-functions ratios conjecture makes predictions on the asymptotic behaviors of the sum of ratios of products of shifted $L$-functions.  This conjecture originated from the work of D. W. Farmer \cite{Farmer93} on shifted moments of the Riemann zeta function, and is formulated for general $L$-functions by J. B. Conrey, D. W. Farmer and M. R. Zirnbauer in \cite[Section 5]{CFZ}. \newline

 There are now several results available in the literature on the ratios conjecture, all valid for certain ranges of the relevant parameters, starting with the work of H. M. Bui, A. Florea and J. P. Keating in \cite{BFK21} over function fields for quadratic $L$-functions. Utilizing the powerful tool of multiple Dirichlet series, M. \v Cech \cite{Cech1} studied the case of quadratic Dirichlet $L$-functions under the assumption of the generalized Riemann hypothesis (GRH).  This was the first result of its type over number fields. \newline

Following the approach of M. \v Cech in \cite{Cech1}, the authors studied the ratios conjecture for quadratic Hecke $L$-functions over the Gaussian field $\mq(i)$ in \cite{G&Zhao14}.  It is the aim of this paper to further apply the method of multiple Dirichlet series to develope the ratios conjecture for quadratic twists of modular $L$-functions. To state our result, we fix a holomorphic Hecke eigenform $f$ of weight $\kappa $ of the full modular group $SL_2 (\mathbb{Z})$.   The Fourier expansion of $f$ at infinity can be written as
\begin{align*}
f(z) = \sum_{n=1}^{\infty} \lambda_f (n) n^{(\kappa -1)/2} e(nz), \quad \mbox{where} \quad e(z)= \exp(2 \pi i z) .
\end{align*}

 We reserve the letter $p$ for a prime throughout the paper. For any Dirichlet character $\chi$ modulo $d$, the twisted modular $L$-function $L(s, f \otimes \chi_d)$ is defined for $\Re(s)>1$ by
\begin{align}
\label{fdfunctioneqn}
L(s, f \otimes \chi) &= \sum_{n=1}^{\infty} \frac{\lambda_f(n)\chi(n)}{n^s}
 = \prod_{p\nmid d} \left(1 - \frac{\lambda_f (p) \chi(p)}{p^s}  + \frac{1}{p^{2s}}\right)^{-1}.
\end{align}
We also write $L(s, \sym f)$ for the symmetric square $L$-function of $f$ defined in \eqref{Lsymexp} and note that $L(s, \sym f)$ is holomorphic for $\Re(s) \geq 1/2$ (see the discussions given in Section \ref{sec2.4}). \newline

Let $\chi^{(m)}=\leg {m}{\cdot}$ denote the Kronecker symbol defined on \cite[p. 52]{iwakow} for any integer $m \equiv 0, 1 \pmod 4$.  Note that every such $m$ factors uniquely into $m=dl^2$, where $d$ is a fundamental discriminant, i.e. $d$ is square-free, $d \equiv 1 \pmod 4$ or $d=4n$ with $n$ square-free, $n \equiv 2,3 \pmod 4$.  Furthermore, when $d$ is a fundamental discriminant, the function $L(s, f \otimes \chi^{(d)})$ has an analytical continuation to the entire complex plane and satisfies the functional equation (see, for example, \cite{S&Y})
\begin{align}
\label{equ:FEfd}
\Lambda (s, f \otimes \chi^{(d)}) =: \left(\frac{|d|}{2\pi} \right)^s \Gamma \left( s + \frac{\kappa -1}{2} \right) L(s, f \otimes \chi^{(d)})
= i^\kappa \epsilon(d ) \Lambda (1- s, f \otimes \chi^{(d)}),
\end{align}
where $\epsilon(d) = 1$ if $d >0$ and $\epsilon(d) = -1$ if $d<0$. \newline

 For an odd, positive integer $n$,  we denote $\chi_n$ for the quadratic character $\left(\frac {\cdot}{n} \right)$.  Given any $L$-function, we write $L^{(c)}$ (resp. $L_{(c)}$) for the function given by the Euler product defining $L$ but omitting those primes dividing (resp. not dividing) $c$. We also write $L_p$ for $L_{(p)}$. We observe that the quadratic reciprocity law implies that $L^{(2)}(s, f \otimes \chi_n)=L(s, f \otimes \chi^{(4n)})$ when $n \equiv 1 \pmod 4$ and that $L^{(2)}(s, f \otimes \chi_n)=L(s, f \otimes \chi^{(-4n)})$ when $n \equiv -1 \pmod 4$. It follows from \eqref{fdfunctioneqn}, \eqref{equ:FEfd} and the above discussions that $L^{(2)}(s, f \otimes \chi_n)$ can also be continued analytically to the entirety of $\comc$.\newline

Our result in this paper investigates the ratios conjecture with one shift in the numerator and denominator for the family of quadratic twist of modular $L$-functions $L^{(2)}(s, f \otimes \chi_n)$ averaged over all odd, positive $n$.
\begin{theorem} \label{Theorem for all characters}
With the notation as above and the truth of GRH, let $w(t)$ be a non-negative Schwartz function and $\widehat w(s)$ its Mellin transform.  For any $\varepsilon>0$, $1/2>\Re(\alpha)>0$, $\Re(\beta)>\varepsilon$, we have
\begin{align}
\label{Asymptotic for ratios of all characters}
\begin{split}	
\sum_{\substack{(n,2)=1}}\frac{L^{(2)}(\tfrac{1}{2}+\alpha, f \otimes \chi_{n})}{L^{(2)}(\tfrac{1}{2}+\beta, f \otimes \chi_{n})}w \bfrac {n}X =&  X\M w(1)L^{(2)}(1+2\alpha, \sym f)P(1,\tfrac{1}{2}+\alpha,\tfrac{1}{2}+\beta;f)+O\lz(1+|\alpha|)^{\varepsilon}|\beta|^\varepsilon X^{N(\alpha,\beta)+\varepsilon}\pz,
\end{split}
\end{align}
 where
\begin{align}
\label{Pdef}
\begin{split}
		P(s, w,z;f)=& \lz 1-\frac 1{2^s} \pz \frac {1}{\zeta^{(2)}(2w)}\\
& \times \prod_{p>2}\lz   1+\frac 1{p^{2z}}  \lz 1-\frac 1{p^s} \pz \lz 1+\frac 1{p^{2w}} \pz +\lz 1-\frac 1{p^s} \pz \frac 1{p^{2w}}-\frac {
\lambda^2_f(p)-2}{p^{2w+s}}+\frac 1{p^{4w+s}}  -\lz1-\frac1{p^s}\pz \frac {\lambda^2_f(p)}{p^{z+w}} \pz,
\end{split}
\end{align}
   and
\begin{equation}
\label{Nab}
			N(\alpha,\beta)=\max\left\{1-2\Re(\alpha),1-2\Re(\beta) \right\}.
\end{equation}
\end{theorem}
	
It follows from our discussions in Section \ref{sec:resA} below that the value $P(1,\tfrac{1}{2}+\alpha,\tfrac{1}{2}+\beta;f)$ is finite for the ranges of $\alpha$ and $\beta$ defined in the statement of Theorem \ref{Theorem for all characters}. The main term in \eqref{Asymptotic for ratios of all characters} is consistent with the ratios conjecture, which can be derived following the treatments given in \cite[Section 5]{G&Zhao2022-2}. However, the error term in \eqref{Asymptotic for ratios of all characters} is inferior to the prediction of the ratios conjecture.  The latter asserts that \eqref{Asymptotic for ratios of all characters} holds uniformly for $|\Re(\alpha)|< 1/4$, $(\log X)^{-1} \ll \Re(\beta) < 1/4$ and $\Im(\alpha), \; \Im(\beta) \ll X^{1-\varepsilon}$ with an error term $O(X^{1/2+\varepsilon})$.  Nevertheless, the advantage of our result in \eqref{Asymptotic for ratios of all characters} is that there is no constraint on imaginary parts of $\alpha$ or $\beta$. \newline

  Theorem \ref{Theorem for all characters} will be established using the method in the proof of \cite[Theorem 1.2]{Cech1}. In particular, we note that the functional equation of a general (not necessarily primitive) quadratic Dirichlet $L$-function given in \cite[Proposition 2.3]{Cech1} plays a crucial role in our proof.

\section{Preliminaries}
\label{sec 2}

We first include some auxiliary results.

\subsection{Quadratic Gauss sums}
\label{sec2.4}
    Recall that $\chi^{(d)}=\leg {d}{\cdot} $ for the Kronecker symbol.  Moreover, let $\psi_j, j \in \{ \pm 1, \pm 2\}$ denote the quadratic characters given by $\psi_j=\chi^{(4j)}$. Note that $\psi_j$ is a primitive character modulo $4j$ for each $j$.  We also denote $\psi_0$ the primitive principle character. \newline
	
  For any integer $q$ and any Dirichlet character $\chi$ modulo $n$, we define the associated Gauss sum $\tau(\chi,q)$ by
\begin{equation*}
		\tau(\chi,q)=\sum_{j\pmod n}\chi(j)e \lz \frac {jq}n \pz.
\end{equation*}

The following result is quoted from \cite[Lemma 2.2]{Cech1}.
\begin{lemma}
\label{Lemma changing Gauss sums}
\begin{enumerate}
\item If $l\equiv1 \pmod 4$, then
\begin{equation*}
		\tau\lz\chi^{(4l)},q\pz=
\begin{cases}
					0,&\hbox{if $(q,2)=1$,}\\
					-2\tau\lz\chi_l,q\pz,&\hbox{if $q\equiv2 \pmod 4$,}\\
					2\tau\lz \chi_l ,q\pz,&\hbox{if $q\equiv0 \pmod 4$.}
\end{cases}
\end{equation*}
			\item If $l \equiv3 \pmod 4$, then
\begin{equation*}
				\tau\lz\chi^{(4l)},q\pz=\begin{cases}
					0,&\hbox{if $2|q$,}\\
					-2i\tau\lz\chi_l,q\pz,&\hbox{if $q\equiv1 \pmod 4$,}\\
					2i\tau\lz\chi_l,q\pz,&\hbox{if $q\equiv3 \pmod 4$.}
				\end{cases}
\end{equation*}
		\end{enumerate}
\end{lemma}
	
	We further define $G\lz\chi_n,q\pz$ by
\begin{align*}
\begin{split}
			G\lz\chi_n,q\pz&=\lz\frac{1-i}{2}+\leg{-1}{n}\frac{1+i}{2}\pz\tau\lz\chi_n,q\pz=\begin{cases}
				\tau\lz\chi_n,q\pz,&\hbox{if $n\equiv1 \pmod 4$,}\\
				-i\tau\lz\chi_n,q\pz,&\hbox{if $n\equiv3\pmod 4$}.
			\end{cases}
\end{split}
\end{align*}
	
We denote $\varphi(m)$ for the Euler totient function of $m$. Our next result is taken from \cite[Lemma 2.3]{sound1} and evaluates $G\lz\chi_m,q\pz$.
\begin{lemma}
\label{lem:Gauss}
   If $(m,n)=1$ then $G(\chi_{mn},q)=G(\chi_m,q)G(\chi_n,q)$. Suppose that $p^a$ is
   the largest power of $p$ dividing $q$ (put $a=\infty$ if $m=0$).
   Then for $k \geq 0$ we have
\begin{equation*}
		G\lz\chi_{p^k},q\pz=\begin{cases}\varphi(p^k),&\hbox{if $k\leq a$, $k$ even,}\\
			0,&\hbox{if $k\leq a$, $k$ odd,}\\
			-p^a,&\hbox{if $k=a+1$, $k$ even,}\\
			\leg{qp^{-a}}{p}p^{a}\sqrt p,&\hbox{if $k=a+1$, $k$ odd,}\\
			0,&\hbox{if $k\geq a+2$}.
		\end{cases}
\end{equation*}
\end{lemma}
	
\subsection{Modular $L$-functions}
\label{sec:cusp form}

For any Dirichlet character $\chi$, the twisted modular $L$-function $L(s, f \otimes \chi)$ has an Euler product for $\Re(s)>1$ given by
\begin{align}
\label{Lfdef}
L(s, f \otimes \chi) &=
\prod_{p} \prod^2_{j =1} (1-\alpha_{f}(p,j)\chi(p)p^{-s})^{-1}.
\end{align}

  By Deligne's proof \cite{D} of the Weil conjecture, we know that
\begin{align}
\label{alpha}
|\alpha_{f}(p,1)|=|\alpha_{f}(p,2)|=1 \quad \mbox{and} \quad \alpha_{f}(p,1)\alpha_{f}(p,2)=1.
\end{align}

  It follows from  this and \eqref{fdfunctioneqn} that $\lambda_f(n)$ is multiplicative and for $\nu \geq 1$, we have
\begin{align}
\label{lambdafpnu}
\begin{split}
 \lambda_{f}(p^{\nu})=& \sum^{\nu}_{j=0}\alpha^{\nu-j}_f(p,1)\alpha^j_{f}(p,2).
\end{split}
\end{align}
  The above relation further implies that $\lambda_f (n) \in \mr$, $\lambda_f (1) =1$ and
\begin{align}
\label{lambdabound}
\begin{split}
  |\lambda_{f}(n)| \leq d(n) \ll n^{\varepsilon},
\end{split}
\end{align}
 where $d(n)$ is the number of divisors of $n$ and the last estimate above follows from \cite[Theorem 2.11]{MVa1}. \newline

We also deduce from \eqref{Lfdef} that for $\Re(s)>1$,
\begin{align}
\begin{split}
\label{Linverse}
 L^{-1}(s, f \otimes \chi)=\prod_p\prod^2_{j=1}(1-\alpha_{f}(p, j)\chi(p)p^{-s})=:\sum^{\infty}_{n=1}\frac {c_{f}(n)\chi(n)}{n^s}.
\end{split}
\end{align}

  It follows from \eqref{alpha}, \eqref{lambdafpnu} and the Euler product given in \eqref{Linverse} that $c_{f}(n)$ is a multiplicative function of $n$ such that
\begin{align}
\begin{split}
\label{aexp}
  c_{f}(p^k)=
\displaystyle
\begin{cases}
 1, \quad k=0, 2, \\
 -\lambda_f(p), \quad k=1, \\
   0, \quad  k>2.
\end{cases}
\end{split}
\end{align}

  In particular, we deduce from \eqref{alpha} and \eqref{aexp} that we have $|c_{f}(p^k)| \leq 2$ for all $k \geq 0$, so that for any
$\varepsilon>0$,
\begin{align}
\label{abound}
\begin{split}
  c_{f}(n) \ll 2^{\omega(n)} \ll n^{\varepsilon},
\end{split}
\end{align}
  where $\omega(n)$ denotes the number of distinct primes dividing $n$ and the last estimation above follows from the well-known bound (see \cite[Theorem 2.10]{MVa1})
\begin{align}
\label{omegabound}
   \omega(h) \ll \frac {\log h}{\log \log h}, \; \mbox{for} \; h \geq 3.
\end{align}

  Recall also that the symmetric square $L$-function $L(s, \operatorname{sym}^2 f)$ of $f$ is defined for $\Re(s)>1$ by
 (see \cite[p. 137]{iwakow} and \cite[(25.73)]{iwakow})
\begin{align}
\label{Lsymexp}
\begin{split}
 L(s, \operatorname{sym}^2 f)=& \prod_p\prod_{1 \leq i \leq j \leq 2} \lz 1-\frac{\alpha_{f}(p, i)\alpha_{f}(p, j)}{p^s} \pz^{-1}=\prod_p \lz 1-\frac{\alpha^2_{f}(p, 1)}{p^s} \pz^{-1} \lz 1-\frac{1}{p^s} \pz^{-1} \lz 1-\frac{\alpha^2_{f}(p, 2)}{p^s} \pz^{-1} \\
    =& \zeta(2s) \sum_{n = 1}^{\infty} \frac {\lambda_f(n^2)}{n^s}=\prod_{p} \left( 1-\frac {\lambda_f(p^2)}{p^s}+\frac {\lambda_f(p^2)}{p^{2s}}-\frac {1}{p^{3s}} \right)^{-1}.
\end{split}
\end{align}

  It follows from a result of G. Shimura \cite{Shimura} that the corresponding completed $L$-function
\begin{align}
\label{Lambdafdef}
 \Lambda(s, \operatorname{sym}^2 f)=& \pi^{-3s/2}\Gamma \left( \frac {s+1}{2} \right)\Gamma \left(\frac {s+\kappa-1}{2} \right) \Gamma \left( \frac {s+\kappa}{2} \right) L(s, \operatorname{sym}^2 f)
\end{align}
  is entire and satisfies the functional equation
$\Lambda(s, \operatorname{sym}^2 f)=\Lambda(1-s, \operatorname{sym}^2 f)$.

\subsection{Functional equations for Dirichlet $L$-functions}
	
	A key ingredient needed in our proof of Theorem \ref{Theorem for all characters} is the following functional equation for all Dirichlet characters $\chi$ modulo $n$ from \cite[Proposition 2.3]{Cech1}.
\begin{lemma}
\label{Functional equation with Gauss sums}
		Let $\chi$ be any Dirichlet character modulo $n \neq \square$ such that $\chi(-1)=1$. Then we have
\begin{equation}
\label{Equation functional equation with Gauss sums}
			L(s,\chi)=\frac{\pi^{s-1/2}}{n^s}\frac{\Gamma\bfrac{1-s}{2}}{\Gamma\bfrac {s}2} K(1-s,\chi), \quad \mbox{where} \quad K(s,\chi)=\sum_{q=1}^\infty\frac{\tau(\chi,q)}{q^s}.
\end{equation}
\end{lemma}

\subsection{Bounding $L$-functions}

In this section, we gather various estimates on the values of $L$-functions under GRH, necessary in the sequal.  For any quadratic character $\psi$ modulo $n$, we write $\widehat{\psi}$ for the primitive character that induces $\psi$. Note that every such $\widehat\psi$ can be written in form $\widehat\psi=\chi^{(d)}$ for some fundamental discriminant $d|n$ (see \cite[Theorem 9.13]{MVa1}). We further write $n$ uniquely as $n=n_1n_2$ with $(n_1, d)=1$ and $p |n_2 \Rightarrow p|d$. Using these notations, we get that for any integer $q$,
\begin{align}
\label{Ldecomp}
\begin{split}
 L^{(q)}( s, f \otimes \psi ) =L( s, f \otimes \widehat{\psi}) \prod_{p|qn_1} \lz 1-\frac {\alpha_f(p,1)\widehat\psi(p)}{p^s} \pz \lz 1-\frac {\alpha_f(p,2)\widehat\psi(p)}{p^s} \pz.
\end{split}
\end{align}

  We now apply \eqref{alpha} to see that for $i=1,2$,
\begin{align*}
 \Big |1-\frac {\alpha_f(p,i)\widehat\psi(p)}{p^s}\Big | \leq 2p^{\max (0,-\Re(s))}.
\end{align*}

  It follows from the above that
\begin{align}
\label{Lnbound}
\begin{split}
 \prod_{p|qn_1} \lz 1-\frac {\alpha_f(p,1)\widehat\psi(p)}{p^s} \pz \lz 1-\frac {\alpha_f(p,2)\widehat\psi(p)}{p^s} \pz \ll 4^{\omega(q_1n)}(qn_1)^{\max (0,-\Re(2s))} \ll (qn_1)^{\max (0,-\Re(2s))+\varepsilon},
\end{split}
\end{align}
  where the last estimate above follows from \eqref{omegabound}. \newline

   It follows from this and the functional equation in \eqref{equ:FEfd} that, for $\Re(s) \leq 1/2$,
\begin{align}
\label{Ldecomp1}
\begin{split}
 L( s, f \otimes \widehat{\psi}) \ll n^{1-2\Re(s)}\frac {\Gamma(1-s+ \frac{\kappa -1}{2})}{\Gamma (s+ \frac{\kappa -1}{2})}L( 1-s, f \otimes \widehat{\psi}) \ll (n(1+|s|))^{1-2\Re (s)}L( 1-s, f \otimes \widehat{\psi}),
\end{split}
\end{align}
  where the last bound above follows from Stirling's formula (see \cite[(5.113)]{iwakow}), which gives that for constants $c_0, d_0$,
\begin{align} \label{Stirlingratio}
  \frac {\Gamma(c_0(1-s)+ d_0)}{\Gamma (c_0s+ d_0)} \ll (1+|s|)^{c_0(1-2\Re (s))}.
\end{align}

   We further note from \cite[Theorem 5.19, Corollary 5.20]{iwakow} that under GRH, we have for  $\Re(s) \geq 1/2$,
\begin{align} \label{PgLest1}
\begin{split}
& \big| L\lz s, f \otimes \widehat\psi \pz \big |, \quad \big| L\lz s,  \widehat\psi \pz \big | \ll |sn|^{\varepsilon}, \quad \big|L\lz s, \sym f \pz \big |  \ll |s|^{\varepsilon}.
\end{split}
\end{align}

  The above combined with \eqref{Ldecomp}--\eqref{PgLest1} now implies that
\begin{align}
\label{Lfchibound1}
\begin{split}
 L^{(q)}( s, f \otimes \psi ) \ll (qn_1)^{\max (0,-\Re(s))+\varepsilon}(n(1+|s|))^{\max \{1-2\Re (s), 0\} +\varepsilon}.
\end{split}
\end{align}

  Our discussions above also apply to other $L$-functions. For example, for the Dirichlet $L$-function $L(s, \psi)$.  Under our notation above, if $d$ is a fundamental discriminant, then we have the following functional equation (see \cite[p. 456]{sound1}) for $L(s, \chi^{(d)})$.
\begin{align*}
  \Lambda(s, \chi^{(d)}) =: \Big(\frac {|d|} {\pi} \Big )^{\frac {s}{2}}\Gamma \Big(\frac {s}{2} \Big)L(s, \chi^{(d)})=\Lambda(1-s,  \chi^{(d)}).
\end{align*}
  It follows from this and \eqref{Stirlingratio} that
\begin{align}
\label{Lchibound1}
\begin{split}
 L^{(q)}( s, \psi ) \ll (qn_1)^{\max (0,-\Re(s))+\varepsilon}(n(1+|s|))^{\max \{1/2-\Re (s),0 \} +\varepsilon}.
\end{split}
\end{align}

  Similarly, by \eqref{Lambdafdef} and \eqref{Stirlingratio}, we get
\begin{align} \label{Lsymbound}
\begin{split}
 L^{(q)}(s, \sym f )  \ll q^{\max (0,-\Re(s))+\varepsilon}(1+|s|)^{\max \{3(1/2-\Re (s)), 0\} +\varepsilon}.
\end{split}
\end{align}

Lastly, from \cite[Theorem 5.19, Corollary 5.20]{iwakow}, we have, under GRH, for  $\Re(s) \geq 1/2+\varepsilon$,
\begin{align} \label{PgLest2}
\begin{split}
&  \big| L( s, f \otimes \widehat\psi ) \big |^{-1} \ll |sn|^{\varepsilon}.
\end{split}
\end{align}

\subsection{Some results on multivariable complex functions}
	
   We include in this section some results from multivariable complex analysis. First we need the notation of a tube domain.
\begin{defin}
		An open set $T\subset\mc^n$ is a tube if there is an open set $U\subset\mr^n$ such that $T=\{z\in\mc^n:\ \Re(z)\in U\}.$
\end{defin}
	
   For a set $U\subset\mr^n$, we define $T(U)=U+i\mr^n\subset \mc^n$.  We have the following Bochner's Tube Theorem \cite{Boc}.
\begin{theorem}
\label{Bochner}
		Let $U\subset\mr^n$ be a connected open set and $f(z)$ be a function holomorphic on $T(U)$. Then $f(z)$ has a holomorphic continuation to the convex hull of $T(U)$.
\end{theorem}

The convex hull of an open set $T\subset\mc^n$ is denoted by $\widehat T$.  Then we quote the result from \cite[Proposition C.5]{Cech1} on the modulus of holomorphic continuations of functions in multiple variables.
\begin{prop}
\label{Extending inequalities}
		Assume that $T\subset \mc^n$ is a tube domain, $g,h:T\rightarrow \mc$ are holomorphic functions, and let $\tilde g,\tilde h$ be their holomorphic continuations to $\widehat T$. If  $|g(z)|\leq |h(z)|$ for all $z\in T$, and $h(z)$ is nonzero in $T$, then also $|\tilde g(z)|\leq |\tilde h(z)|$ for all $z\in \widehat T$.
\end{prop}

\section{Proof of Theorem \ref{Theorem for all characters}}

Let $\mu$ denote the M\"obius function. Using the notations defined in Section \ref{sec:cusp form}, we define for $\Re(s), \Re(w), \Re(z)$ large enough,
\begin{align}
\label{Aswzexp}
\begin{split}
A(s,w,z;f)= \sum_{\substack{(n,2)=1 }}\frac{L^{(2)}(w, f \otimes \chi_{n})}{L^{(2)}(z,f \otimes \chi_{n})n^s}
=\sum_{\substack{(nmk,2)=1}}\frac{\lambda_f(m)c_f(k)\chi_n(k)\chi_n(m)}{k^zm^wn^s} = \sum_{\substack{(mk,2)=1}}\frac{\lambda_f(m)c_f(k)L\lz s,\chi^{(4mk)} \pz}{m^wk^z}.
\end{split}
\end{align}

We shall devote the next few sections to articulating the analytical properties of $A(s,w,z; f)$ as the proof of Theorem \ref{Theorem for all characters} relies crucially on them.
	
\subsection{First region of absolute convergence of $A(s,w,z;f)$}

   We start with the series representation for $A(s,w,z;f)$ given by the first equality in \eqref{Aswzexp}.  This gives that if $\Re(z)>1/2$,
\begin{align}
\begin{split}
\label{Abound}
		A(s,w,z;f)=& \sum_{\substack{(n,2)=1}}\frac{L^{(2)}(w,f \otimes \chi_n)}{L^{(2)}(z, f \otimes \chi_n)n^s}\\
=& \sum_{\substack{(h,2)=1}}\frac {1}{h^{2s}}\sumstar_{\substack{(n,2)=1}}\frac{L^{(2)}(w, f \otimes \chi_{n})\prod_{p | h}(1-\alpha_f(p,1)\chi_{n}(p)p^{-w})(1-\alpha_f(p,2)\chi_{n}(p)p^{-w}) }{n^s L^{(2)}(z, f \otimes \chi_{n})\prod_{p | h}(1-\alpha_f(p,1)\chi_{n}(p)p^{-z})(1-\alpha_f(p,2)\chi_{n}(p)p^{-z})} \\
\ll& \sum_{\substack{(h,2)=1}}\frac {h^{\max (0,-2\Re(w))+\varepsilon} }{h^{2s}}\sumstar_{\substack{(n,2)=1}}\Big | \frac{L^{(2)}(w, f \otimes \chi_n)}{L^{(2)}(z,f \otimes \chi_{n})n^s} \Big |,
\end{split}
\end{align}
  where $\sum^*$ henceforth denotes the sum over square-free integers and the last estimate above is obtained using \eqref{Lnbound} and a similar observation that if $\Re(z)>1/2$,
\begin{align} \label{prodinvest}
\begin{split}
		\prod_{p | h}(1-\alpha_f(p,1)\chi_{n}(p)p^{-z})^{-1}(1-\alpha_f(p,2)\chi_{n}(p)p^{-z})^{-1} \ll h^{\varepsilon}.
\end{split}
\end{align}

  Recall that we have $L^{(2)}(w, f \otimes \chi_n)=L(w, f \otimes \chi^{(\pm 4n)})$ for $n \equiv \pm 1 \pmod 4$.  We write $\widetilde \chi_n$ for the primitive Dirichlet character that induces $\chi^{(\pm 4n)}$. For a square-free $n$, we see that $\widetilde \chi_n=\chi^{(n)}$ is a primitive character modulo $n$ if $n \equiv 1 \pmod 4$ and $\widetilde \chi_n=\chi^{(4n)}$ is a primitive character modulo $4n$ if $n \equiv -1 \pmod 4$.  For $n \equiv 1 \pmod 4$, we have
\begin{align*}
\begin{split}
	\big|L^{(2)}(w, f \otimes \chi_n)\big|=&	\big|(1-\alpha_f(2,1)\widetilde\chi_n(2)2^{-w})(1- \alpha_f(2,2)\widetilde\chi_n(2)2^{-w})L(w,f \otimes \widetilde\chi_n)\big| \ll  \big|L(w,f \otimes \widetilde\chi_n)\big |.
\end{split}
\end{align*}
 The above bound also holds for $n \equiv -1 \pmod 4$. \newline

   Similarly, if $\Re(z)>1/2$, we have, by \eqref{PgLest2} and an estimation analogous to \eqref{prodinvest}, that under GRH,
\begin{align*}
\begin{split}
	|L^{(2)}(z, f \otimes \chi_n)|^{-1} \ll n^{\varepsilon}|L(z,f \otimes \widetilde\chi_n)|^{-1} \ll (|z|n)^{\varepsilon}.
\end{split}
\end{align*}

The above, together with \eqref{Abound}, allows us to deduce that if $\Re(z) >1/2$, then
\begin{align}
\begin{split}
\label{Aboundinitial}
		A(s,w,z;f)
\ll& \sum_{\substack{(h,2)=1}}\frac {h^{\max (0,-2\Re(w))+\varepsilon} }{h^{2s}}\sumstar_{\substack{(n,2)=1}}\frac{|L(w, f \otimes \widetilde\chi_n)|}{|n^{s-\varepsilon}|}.
\end{split}
\end{align}

   We now apply \eqref{Lfchibound1} to deduce from the above that under GRH, except for a simple pole at $w=1$, both sums of the right-hand side expression in \eqref{Aboundinitial} are convergent for $\Re(s)>1$,  $\Re(w) \geq 1/2$, $\Re(z)>1/2$ as well as for $\Re(2s)>1$, $\Re(2s+2w)>1$, $\Re(s+2w)>2$, $\Re(w) < 1/2$, $\Re(z)>1/2$. \newline

  We thus conclude that the function $A(s,w,z;f)$ converges absolutely in the region
\begin{equation*}
		S_0=\{(s,w,z): \Re(s)>1,\ \Re(2s+2w)>1,\ \Re(s+2w)>2,\ \Re(z)> \tfrac{1}{2} \}.
\end{equation*}

     Note that the condition $\Re(2s+2w)>1$ is implied by the other conditions so that we have
\begin{equation*}
		S_0=\{(s,w,z): \Re(s)>1, \ \Re(s+2w)>2,\ \Re(z)> \tfrac{1}{2} \}.
\end{equation*}

Next, we deduce from the last expression of \eqref{Aswzexp} that $A(s,w,z;f)$ is given by the series
\begin{align}
\label{Sum A(s,w,z) over n}		
A(s,w,z;f)=&\sum_{\substack{(mk,2)=1}}\frac{\lambda_f(m)c_f(k) L( s, \chi^{(4mk)})}{m^wk^z}.
\end{align}

We write $mk=(mk)_0(mk)^2_1$ with $(mk)_0$ odd and square-free. Note that $\chi^{((mk)_0)}$ is a primitive character modulo $(mk)_0$ when $(mk)_0 \equiv 1 \pmod 4$.   We now apply \eqref{lambdabound}, \eqref{abound} and \eqref{Lchibound1} to see that, except for a simple pole at $s=1$ arising from the summands with $mk=\square$, the sums over $m,k$ such that $(mk)_0 \equiv 1 \pmod 4$ in \eqref{Sum A(s,w,z) over n} converge absolutely in the region
\begin{align*}
		S_1=& \{(s,w,z): \Re(w)>1,\ \Re(z)>1,\ \Re(s)\geq \tfrac 12 \} \bigcup \{(s,w,z):  0 \leq \Re(s)< \tfrac 12, \ \Re(s+w)>\tfrac32,\ \Re(s+z)>\tfrac32 \} \\
   & \hspace*{3cm} \bigcup \{(s,w,z): \Re(s)<0, \ \Re(2s+w)>\tfrac32,\ \Re(2s+z)>\tfrac32 \}.
\end{align*}

  Note that similar estimations hold when $(mk)_0 \equiv 2,3 \pmod 4$, in which case $\chi^{(4(mk)_0)}$ is a primitive character modulo $4(mk)_0$.  We thus conclude that, except for a simple pole at $s=1$ arising from the summands with $mk=\square$, the function $A(s,w,z;f)$ converges absolutely in the region $S_1$. \newline

 To determine the convex hull of $S_0$ and $S_1$, we first note that for a fixed $\Re(z_0)>1$, the points $(1/2, 1, z_0)$ and $(1, 1/2, z_0)$ are in the closures of $S_0$ and $S_1$, respectively. These two points determine a line segment: $\Re(s+w)=3/2$ with $1/2 \leq \Re(s) \leq 1$ on the plane $\Re(z)=\Re(z_0)$. Note further that when $\Re(s)>1/2$ and $\Re(z)>1$, the conditions $\Re(s+z)>3/2$ and $\Re(2s+z)>3/2$ are automatically satisfied. We then deduce that the convex hull of $S_0$ and $S_1$ contains points $(s,w,z)$ satisfying
\begin{equation*}
		\{\Re(z)>1, \ \Re(s)>1/2, \ \Re(s+2w)>2,\ \Re(s+z)> \tfrac{3}{2}, \ \Re(s+w)>\tfrac32, \ \Re(2s+z)>\tfrac32\}.
\end{equation*}
  Combining the above set with the subsets of $S_1$ containing points with $\Re(s)<1/2$, we get that the convex hull of $S_0$ and $S_1$ contains points $(s,w,z)$ satisfying
\begin{equation}
\label{ConvexhuallA(s,w,z)}
		\{\Re(z)>1, \Re(s+2w)>2,\ \Re(s+z)> \tfrac{3}{2}, \ \Re(s+w)>\tfrac32, \ \Re(2s+w)>\tfrac32,\ \Re(2s+z)>\tfrac32\}.
\end{equation}
On the other hand, if $1/2 < \Re(z)< 1$, the points in $S_0$ are certainly contained in the convex hull of $S_0$ and $S_1$. We may thus focus on the case $\Re(s)<1$.  In this case, we note that the points $(1,1/2, 1/2)$, $(1,1/2,1)$ are in the closure of $S_0$ and $(1/2,1,1)$ the closure of $S_1$. These three points determine a triangular region which can be regarded as the base of the region $\mathcal R$ enclosed by the four planes: $\Re(s+w)=3/2$, $\Re(s)=1$, $\Re(z)=1$, $\Re(s+z)=3/2$. It follows that the points in this triangle region are all in the convex hulls of $S_0$ and $S_1$.  Further note that the points on the boundary $\mathcal R \cap \{\Re(z)=1\}$ of $\mathcal R$ are all in the convex hulls of $S_0$ and $S_1$ since they can be identified with the set $\{(s, w, z): \Re(z)=1,  1/2 \leq \Re(s) \leq 1, \Re(s+w) \geq 3/2\}$ and hence is contained in the convex hull of the set given in \eqref{ConvexhuallA(s,w,z)}.  Hence the entire region $\mathcal R$ lies in the convex hull of $S_0$ and $S_1$.  Next, if $1/2 < \Re(z)< 1$, the condition $\Re(s+z)>3/2$ implies that $\Re(s)>1/2$ so that one also has $\Re(2s+z)>3/2$. Similarly, when $1/2 < \Re(z)< 1$, the condition $\Re(s+w)>3/2$ implies that $\Re(2s+w)>3/2$. Lastly, the condition $\Re(s+2w)>2$ implies $\Re(s+w)>3/2$ and $\Re(2s+w)>3/2$ for $\Re(s)>1$.  It follows from the discussions here that the intersection of the convex hulls of $S_0$ and $S_1$ thus equals
\begin{equation}
\label{Region of convergence of A(s,w,z)}
		S_2=\{(s,w,z):\Re(z)> \tfrac{1}{2},\ \Re(s+2w)>2,\ \Re(s+z)> \tfrac{3}{2}, \ \Re(s+w)>\tfrac32, \ \Re(2s+w)>\tfrac32,\ \Re(2s+z)>\tfrac32\}.
\end{equation}

The above, together with Theorem~\ref{Bochner}, implies that $(s-1)(w-1)A(s,w,z;f)$ converges absolutely in the region $S_2$.

\subsection{Residue of $A(s,w,z;f)$ at $s=1$} \label{sec:resA}

It follows from \eqref{Sum A(s,w,z) over n} that $A(s,w,z;f)$ has a pole at $s=1$ arising from the terms with $mk=\square$. To compute the corresponding residue and for our treatments later in the proof, we introduce the sum
\begin{align*}
\begin{split}
 A_1(s,w,z;f) =: \sum_{\substack{(mk,2)=1 \\ mk =  \square}}\frac{\lambda_f(m)c_f(k)L\lz s, \chi^{(4mk)}\pz}{m^wk^z}
=  \sum_{\substack{(mk,2)=1 \\ mk =  \square}}\frac{\lambda_f(m)c_f(k)\zeta(s)\prod_{p | 2mk}(1-p^{-s}) }{m^wk^z} .
\end{split}
\end{align*}

 We further denote by $a_t(n)$ for any $t \in \mc$ the multiplicative function such that $a_t(p^k)=1-1/p^t$ for any prime $p$. Thus, we recast $A_1(s,w,z;f)$ as
\begin{align*}
\begin{split}
 A_1(s,w,z;f)
=  \zeta^{(2)}(s)\sum_{\substack{(mk,2)=1 \\ mk =  \square}}\frac{\lambda_f(m)c_f(k)a_s(mk) }{m^wk^z} .
\end{split}
\end{align*}

We now write the last sum above as an Euler product.  Slightly abusing notation by writing $p^{k'}$ with $k' \in \intz$ for the highest power of $p$ dividing $k$, and similarly for $m$. Thus, we obtain, utilizing \eqref{aexp}, that
\begin{align} \label{residuesgen}
\begin{split}
	 A_1& (s,w,z;f) \\
=&  \zeta^{(2)}(s)\prod_{p>2}\sum_{\substack{m,k\geq0\\m+k\even}}\frac{\lambda_f(p^m)c_f(p^k)a_s(p^{m+k})}{p^{mw+kz}}\\
=& \zeta^{(2)}(s)\prod_{p>2}\lz\sum_{\substack{m\geq0\\ 2 \mid m}}\frac{\lambda_f(p^m) a_s(p^m)}{p^{mw}}+\sum_{\substack{m\geq0\\ (m,2)=1}}\frac{c_f(p)\lambda_f(p^m)a_s(p^{m+1})}{p^{z+mw}}+\sum_{\substack{m\geq0\\ 2 \mid m}}\frac{c_f(p^2)\lambda_f(p^m)a_s(p^{m+2})}{p^{2z+mw}}\pz\\
=& \zeta^{(2)}(s)\prod_{p>2}\lz1+\lz 1-\frac 1{p^s} \pz \frac 1{p^{2z}}+\lz 1-\frac 1{p^s} \pz\lz 1+\frac 1{p^{2z}} \pz \sum^{\infty}_{m=1}\frac {\lambda_f(p^{2m})}{p^{2mw}} -\lz1-\frac1{p^s}\pz \frac {\lambda_f(p)}{p^z}\sum^{\infty}_{m=0}\frac {\lambda_f(p^{2m+1})}{p^{(2m+1)w}}  \pz.
\end{split}
\end{align}

  We now use \eqref{lambdafpnu} to infer that (we may assume that $\alpha_f(p,1) \neq \alpha_f(p,2)$ as the other case follows from continuity)
\begin{align}
\label{lambdasum1}
\begin{split}
   \sum^{\infty}_{i=0}\frac {\lambda_f(p^{2i})}{p^{iu}}=&  \sum^{\infty}_{i=0}\frac {1}{p^{iu}}\frac {\alpha^{2i+1}_f(p,1)-\alpha^{2i+1}_f(p,2)}{\alpha_f(p,1)-\alpha_f(p,2)}=  \lz 1-\frac{\alpha^2_f(p,1)}{p^u} \pz^{-1} \lz 1-\frac{\alpha^2_f(p,2)}{p^u} \pz^{-1} \lz 1+ \frac{1}{p^u} \pz.
\end{split}
\end{align}

   Similarly, we have
\begin{align}
\label{lambdasum2}
\begin{split}
 \sum^{\infty}_{i=0}\frac {\lambda_f(p^{2i+1})}{p^{(i+1)u}}
=& \lz 1-\frac{\alpha^2_f(p,1)}{p^u} \pz^{-1} \lz 1-\frac{\alpha^2_f(p,2)}{p^u} \pz^{-1}\frac {\lambda_f(p)}{p^{u}}.
\end{split}
\end{align}

Inserting \eqref{lambdasum1} and \eqref{lambdasum2} into \eqref{residuesgen} yields
\begin{align}
\label{residuesgensimplified}
\begin{split}
A_1&(s,w,z;f) =  \zeta^{(2)}(s)\prod_{p>2}\sum_{\substack{m,k\geq0\\m+k\even}}\frac{\lambda_f(p^m)c_f(p^k)a_s(p^{m+k})}{p^{mw+kz}} \\
=& \zeta^{(2)}(s)\prod_{p>2} \lz 1-\frac{\alpha^2_f(p,1)}{p^{2w}} \pz^{-1} \lz 1-\frac{\alpha^2_f(p,2)}{p^{2w}} \pz^{-1} \\
& \times \prod_{p>2}\lz \lz 1-\frac 1{p^s} \pz \lz 1+\frac 1{p^{2z}} \pz \lz 1+\frac 1{p^{2w}} \pz +\frac 1{p^s}  \lz 1-\frac{\alpha^2_f(p,1)}{p^{2w}} \pz \lz 1-\frac{\alpha^2_f(p,2)}{p^{2w}} \pz -\lz1-\frac1{p^s}\pz \frac {\lambda^2_f(p)}{p^{z+w}} \pz \\
=& \frac {\zeta^{(2)}(s)  L^{(2)}(2w, \sym f)}{\zeta^{(2)}(2w)} \\
& \times \prod_{p>2}\lz \lz 1-\frac 1{p^s} \pz \lz 1+\frac 1{p^{2z}} \pz \lz 1+\frac 1{p^{2w}} \pz+\frac 1{p^s} \lz 1-\frac{\alpha^2_f(p,1)}{p^{2w}} \pz \lz 1-\frac{\alpha^2_f(p,2)}{p^{2w}} \pz -\lz1-\frac1{p^s}\pz \frac {\lambda^2_f(p)}{p^{z+w}}\pz \\
=& \frac {\zeta^{(2)}(s) L^{(2)}(2w, \sym f)}{\zeta^{(2)}(2w)} \\
& \times \prod_{p>2}\lz   1+\frac 1{p^{2z}}  \lz 1-\frac 1{p^s} \pz \lz 1+\frac 1{p^{2w}} \pz +\lz 1-\frac 1{p^s} \pz \frac 1{p^{2w}}-\frac {
\alpha^2_f(p,1)+\alpha^2_f(p,2)}{p^{2w+s}}+\frac 1{p^{4w+s}}  -\lz1-\frac1{p^s}\pz \frac {\lambda^2_f(p)}{p^{z+w}} \pz \\
=& \zeta(s)L^{(2)}(2w, \sym f)P(s,w,z;f),
\end{split}
\end{align}
  where $P(s,w,z;f)$ is defined in \eqref{Pdef}. \newline

  We deduce from \eqref{Pdef} and \eqref{residuesgensimplified} that except for a simple pole at $s=1$, $P(s, w,z;f)$ and $A_1(s,w,z;f)$ are holomorphic in the region
\begin{align}
\label{S3}
	S_3= \{(s,w,z):\ &\Re(2z)>1, \ \Re(s+2z)>1, \ \Re(s+2w)>1,\ \Re(w+z)>1, \ \Re(4w)>1, \ \Re(s+w+z)>1 \}.
\end{align}

Recalling that the residue of $\zeta(s)$ at $s = 1$ is $1$, we arrive at
\begin{align} \label{Residue at s=1}
 \res_{s=1}& A(s, \tfrac{1}{2}+\alpha,\tfrac{1}{2}+\beta;f) = \res_{s=1} A_1(s, \tfrac{1}{2}+\alpha,\tfrac{1}{2}+\beta;f)
=L^{(2)}(1+2\alpha, \sym f)P(1,\tfrac{1}{2}+\alpha,\tfrac{1}{2}+\beta;f).
\end{align}

\subsection{Second region of absolute convergence of $A(s,w,z;f)$}

From \eqref{Sum A(s,w,z) over n}, we get
\begin{align}
\begin{split}
\label{A1A2}
 A(s,w,z;f) =& \sum_{\substack{(mk,2)=1 \\ mk =  \square}}\frac{\lambda_f(m)c_f(k)L\lz s, \chi^{(4mk)}\pz}{m^wk^z} +\sum_{\substack{(mk,2)=1 \\ mk \neq \square}}\frac{\lambda_f(m)c_f(k)L\lz s, \chi^{(4mk)}\pz}{m^wk^z} \\
=&  \sum_{\substack{(mk,2)=1 \\ mk =  \square}}\frac{\lambda_f(m)c_f(k)\zeta(s)\prod_{p | 2mk}(1-p^{-s}) }{m^wk^z} +\sum_{\substack{(mk,2)=1 \\ mk \neq \square}}\frac{\lambda_f(m)c_f(k)L\lz s, \chi^{(4mk)}\pz}{m^wk^z} \\
=: \ & A_1(s,w,z;f)+A_2(s,w,z;f).
\end{split}
\end{align}
We recall from our discussions in the previous section that except for a simple pole at $s=1$, $A_1(s,w,z;f)$ is holomorphic in the region $S_3$. \newline

  Next, we apply the functional equation given in Lemma \ref{Functional equation with Gauss sums} for $L\lz s, \chi^{(4mk)}\pz$ in the case $mk \neq \square$ by observing that $\chi^{(4mk)}$ is a Dirichlet character modulo $4mk$ for any $m,k\geq1$ with $\chi^{(4mk)}(-1)=1$. We obtain from \eqref{Equation functional equation with Gauss sums} that
\begin{align}
\begin{split}
\label{Functional equation in s}
 A_2(s,w,z;f) =\frac{\pi^{s-1/2}}{4^s}\frac {\Gamma (\frac{1-s}2)}{\Gamma(\frac {s}2) } C(1-s,s+w,s+z;f),
\end{split}
\end{align}
  where $C(s,w,z;f)$ is given by the triple Dirichlet series
\begin{align*}
		C(s,w,z;f)=& \sum_{\substack{q, m, k \\ (mk,2)=1 \\ mk \neq \square}}\frac{\lambda_f(m)c_f(k)\tau(\chi^{(4mk)}, q)}{q^sm^wk^z} \\
=& \sum_{\substack{q, m, k \\ \substack{(mk,2)=1}}}\frac{\lambda_f(m)c_f(k)\tau(\chi^{(4mk)}, q)}{q^sm^wk^z}-\sum_{\substack{q, m, k \\ (mk,2)=1 \\ mk = \square}}\frac{\lambda_f(m)c_f(k)\tau(\chi^{(4mk)}, q)}{q^sm^wk^z}.
\end{align*}	

 By \eqref{Region of convergence of A(s,w,z)}, \eqref{S3} and the functional equation \eqref{Functional equation in s}, we see that $C(s,w,z;f)$ is initially defined for $\Re(s)$, $\Re(z)$ and $\Re(w)$ sufficiently large.  To extend this region, we exchange the summations in $C(s,w,z;f)$ and set $mk=l$ to obtain that
\begin{align}
\label{Cexp}
\begin{split}
  C(s,w,z;f)=& \sum^{\infty}_{q =1}\frac{1}{q^s}\sum_{\substack{(l,2)=1}}\frac{\tau\lz \chi^{(4l)}, q \pz r(l, z-w)}{l^w}-\sum^{\infty}_{q =1}\frac{1}{q^s}\sum_{\substack{(l,2)=1 \\ l = \square}}\frac{\tau\lz \chi^{(4l)}, q \pz r(l, z-w)}{l^w}\\
=: \ & C_1(s,w,z;f)-C_2(s,w,z;f),
\end{split}
\end{align}
  where
\begin{align}
\begin{split}
\label{rdef}
	r(l, t)=\sum_{\substack{ k|l}}\frac{\lambda_f(l/k)c_f(k)}{k^{t}}.
\end{split}
\end{align}

   We now define, for two Dirichlet characters $\psi$ and $\psi'$ whose conductors divide $8$,
\begin{align}
\begin{split}
\label{C12def}
	C_1(s,w,z;\psi,\psi',f)=: \ & \sum_{l,q\geq 1}\frac{G\lz \chi_l,q\pz\psi(l)\psi'(q) r(l, z-w)}{l^wq^s}, \quad \mbox{and} \\
C_2(s,w,z;\psi,\psi',f)=: \ &  \sum_{l,q\geq 1}\frac{G \lz \chi_{l^2},q\pz\psi(l)\psi'(q) r(l^2, z-w)}{l^{2w}q^s}.
\end{split}
\end{align}

  Following the arguments contained in \cite[\S 6.4]{Cech1} and making use of Lemma \ref{Lemma changing Gauss sums}, we see that
\begin{align}
\begin{split}
\label{C(s,w,z) as twisted C(s,w,z)}
		C_1(s,w,z;f)=&
			-2^{-s}\big( C_1(s,w,z;\psi_2,\psi_1,f)+C_1(s,w,z;\psi_{-2},\psi_1,f)\big) +4^{-s}\big( C_1(s,w,z;\psi_1,\psi_0,f) \\
			& \hspace*{1.5cm} +C_1(s,w,z;\psi_{-1},\psi_0,f)\big) +C_1(s,w,z;\psi_1,\psi_{-1},f)-C_1(s,w,z;\psi_{-1},\psi_{-1},f), \quad \mbox{and} \\
   C_2(s,w,z;f)=&
			-2^{1-s}C_2(s,w,z;\psi_1,\psi_1,f)+2^{1-2s}C_2(s,w,z;\psi_1,\psi_0,f).
\end{split}
\end{align}

  Note that every integer $q \geq 1$ can be written uniquely as $q=q_1q^2_2$ with $q_1$ square-free. We may thus write
\begin{equation}
\label{Cidef}
		C_i(s,w,z;\psi,\psi',f)=\sumstar_{q_1}\frac{\psi'(q_1)}{q_1^s}\cdot D_i(s, w,z-w;q_1, \psi,\psi',f), \quad i =1,2,
\end{equation}
	where
\begin{align}
\label{Didef}
\begin{split}
		D_1(s, w,t;q_1,\psi,\psi',f)=&\sum_{l,q_2=1}^\infty\frac{G\lz \chi_{l},q_1q^2_2\pz\psi(l)\psi'(q^2_2) r(l, t)}{l^wq^{2s}_{2}}, \quad \mbox{and} \\
D_2(s, w,t;q_1,\psi,\psi',f)=&\sum_{l,q_2=1}^\infty\frac{G\lz \chi_{l^2},q_1q^2_2\pz\psi(l)\psi'(q^2_2) r(l^2, t)}{l^{2w}q^{2s}_{2}}.
\end{split}
\end{align}

  We have the following result for the analytic properties of $D_i(s, w,t; q_1, \psi, \psi', f)$.
\begin{lemma}
\label{Estimate For D(w,t)}
 With the notation as above, for $\psi \neq \psi_0$, the functions $D_i(s, w,t; q_1,\psi,\psi',f), i=1,2$ have meromorphic continuations to the region
\begin{equation*}
		\{(s,w,t): \Re(s)>1, \ \Re(w)>1,\ \Re(w+t)>1\}.
\end{equation*}
	  For $\Re(s)>1+\varepsilon,  \Re(w)>1+\varepsilon$ and $\Re(w+t)>1+\varepsilon$, we have
\begin{align}
\label{Diest}
			|D_i(s, w,t;q_1, \psi,\psi', f)|\ll |q_1w(t+w)|^{\varepsilon}.
\end{align}		
\end{lemma}
\begin{proof}
   As the proofs are similar, we consider only the case for $D_1(s, w,t;q_1, \psi,\psi',f)$ here.  First $D_1(s,w,t;q_1, \psi,\psi',f)$ are jointly multiplicative functions of $l,q_2$ by Lemma \ref{lem:Gauss} in the double sum defining $D_1$ in \eqref{Didef}. Moreover, as $\psi \neq \psi_0$, we may assume that $l$ is odd. We write $D_1(s,w,z;q_1, \psi,\psi',f)$ using \eqref{aexp} into an Euler product so that
\begin{align}
\label{D1Eulerprod}
\begin{split}	
 &	D_1(s, w,t;q_1, \psi,\psi',f)= \prod_p D_{1,p}(s, w,t;q_1, \psi,\psi',f).
\end{split}
\end{align}
	Then we have
\begin{align}
\label{Dexp}
\begin{split}	
 &	D_{1,p}(s, w,t; q_1, \psi,\psi',f)= \displaystyle
\begin{cases}
\displaystyle \sum_{l=0}^\infty\frac{ \psi'(2^{2l})}{2^{2ls}}, & p=2, \\
\displaystyle \sum_{l,k=0}^\infty\frac{ \psi(p^l)\psi'(p^{2k})G\lz \chi_{p^l}, q_1p^{2k} \pz r(p^l, t) }{p^{lw+2ks}}, & p>2.
\end{cases}
\end{split}
\end{align}
	
  Note that for a fixed $p > 2$,
\begin{align} \label{Dgenest}	
\begin{split}
 \sum_{l,k=0}^\infty & \frac{ \psi(p^l)\psi'(p^{2k})G\lz \chi_{p^l}, q_1p^{2k} \pz r(p^l, t)}{p^{lw+2ks}}  \\
 & = \sum_{l=0}^\infty \frac{ \psi(p^l)G\lz \chi_{p^l}, q_1 \pz r(p^l, t)}{p^{lw}}  + \sum_{l \geq 0, k \geq 1}\frac{ \psi(p^l)\psi'(p^{2k})G\lz \chi_{p^l}, q_1p^{2k} \pz r(p^l, t)}{p^{lw+2ks}}.
\end{split}
\end{align}

  Observe that
\begin{align}
\label{rexp}
\begin{split}	
 r(p^l, t)=\lambda_f(p^l)-\frac{\lambda_f(p^{l-1})\lambda_f(p)}{p^{t}}+\frac{\lambda_f(p^{l-2})}{p^{2t}},
\end{split}
\end{align}
  where we denote $\lambda_f(p^{-i})=0$ for integers $i \geq 0$. \newline

  Notice that by \eqref{alpha} and \eqref{lambdafpnu}, we have
\begin{align}
\label{lambdadest}
\begin{split}
		|\lambda_f(p^l)| \leq l+1, \quad l \geq 0.
\end{split}
\end{align}
 It follows from this and \eqref{rexp} that for $l \geq 1$,
\begin{align*}
\begin{split}
		|r(p^l, t)| \leq
\displaystyle
\begin{cases}
 4(l+1)(1+p^{-t}),  & l=1,  \\
 4(l+1)(1+p^{-2t}), & l \geq 2.
\end{cases}
\end{split}
\end{align*}

    Also, note that Lemma \ref{lem:Gauss} implies that
\begin{align*}
\begin{split}
	|G( \chi_{p^l}, q_1p^{2k} )| \ll p^l.
\end{split}
\end{align*}

  We apply the above estimations to see that when $\Re(s)> 1/2$, $\Re(w)>1$, $\Re(w+t)>1$,
\begin{align} \label{Dk1est}	
\begin{split}
\sum_{l \geq 0, k \geq 1} & \frac{ \psi(p^l)\psi'(p^{2k})G\lz \chi_{p^l}, q_1p^{2k} \pz r(p^l, t) }{p^{lw+2ks}}
= \sum_{k \geq 1}\frac{ \psi'(p^{2k})G\lz \chi_{1}, q_1p^{2k} \pz }{p^{2ks}}+\sum_{l, k \geq 1}\frac{ \psi(p^l)\psi'(p^{2k})G\lz \chi_{p^l}, q_1p^{2k} \pz  r(p^l, t)}{p^{lw+2ks}}  \\
& \ll  p^{-2s} + p^{-2s} \lz \frac{1}{p^{w-1}}4(l+1)(1+p^{-t})+\sum_{l \geq 2}\frac{1}{p^{l(w-1)}}4(l+1)(1+p^{-2t}) \pz  \\
& \ll p^{-2s}+p^{-2s-w+1}+p^{-2s-2w+2}+p^{-2s-w-t+1}+p^{-2s-2w-2t+2} \ll p^{-2s}+p^{-2s-w-t+1}.
\end{split}
\end{align}

  We now apply Lemma \ref{lem:Gauss} and \eqref{rexp} to see that when $p \nmid 2q_1$ and $\Re(w)>1$,
\begin{align}
\label{Dgenl0gen}	
\begin{split}
  \sum_{l=0}^\infty & \frac{ \psi(p^l)G\lz \chi_{p^l}, q_1 \pz r(p^l, t)}{p^{lw}} = 1+\frac{ \psi(p)\lambda_f(p)\chi^{(q_1)}(p)}{p^{w-1/2}}\lz 1-\frac{1}{p^{t}}\pz \\
=& L_p \lz w-\tfrac{1}{2}, f \otimes \chi^{(q_1)}\psi\pz  \lz 1-\frac {\lambda_f^2(p)}{p^{2w-1}}-\frac{\chi^{(q_1)}(p)\psi(p)\lambda_f(p)}{p^{w+t-1/2}}+\frac{\lambda_f^2(p)}{p^{2w+t-1}}\pz
\\
= & \frac {L_{p}\lz w-\tfrac{1}{2},f \otimes \chi^{(q_1)}\psi\pz}{L_{p}(2w-1, \sym f)\zeta_{p}(2w-1)} \\
 & \hspace*{3cm} \times \lz 1-\frac{\chi^{(q_1)}(p) \psi(p)\lambda_f(p)}{p^{w+t-1/2}}+O \Big (\frac{1}{p^{2w+t-1}}+\frac{1}{p^{w+t-1/2+2w-1}}+\frac{1}{p^{2(2w-1)}}\Big ) \pz \\
=& \frac {L_{p}\lz w-\tfrac{1}{2}, f \otimes \chi^{(q_1)}\psi\pz}{L_{p}(2w-1, \sym f)\zeta_{p}(2w-1)L_{p}\lz t+w-1/2,f \otimes \chi^{(q_1)}\psi\pz} \\
& \hspace*{3cm} \times \lz 1+O \Big (\frac{1}{p^{2w+2t-1}}+\frac{1}{p^{2w+t-1}}+\frac{1}{p^{w+t-1/2+2w-1}}+\frac{1}{p^{2(2w-1)}} \Big ) \pz .
\end{split}
\end{align}

  We deduce from \eqref{Dexp}, \eqref{Dgenest}, \eqref{Dk1est} and \eqref{Dgenl0gen} that for $p \nmid 2q_1$, $\Re(s)>\frac 12, \Re(w)>1, \Re(w+t)>1$,
\begin{align}
\label{Dgenexp}	
\begin{split}
 D_{1,p} & (s, w, t;q_1, \psi,\psi',f) \\
= & \frac {L_{p}\lz w-\tfrac{1}{2}, f \otimes \chi^{(q_1)}\psi\pz}{L_{p}(2w-1, \sym f)\zeta_{p}(2w-1)L_{p}\lz t+w-1/2,f \otimes \chi^{(q_1)}\psi\pz} \\
& \hspace*{1cm} \times \lz 1+O \Big (p^{-(2w+2t-1)}+p^{-(2w+t-1)}+p^{-(w+t-1/2+2w-1)}+p^{-2(2w-1)}+ p^{-2s}+p^{-2s-w-t+1}\Big ) \pz .
\end{split}
\end{align}
  The first assertion of the lemma now follows from \eqref{D1Eulerprod}, \eqref{Dexp} and \eqref{Dgenexp}. \newline

  We next note that Lemma \ref{lem:Gauss}, \eqref{rdef} and \eqref{lambdadest} implies that when $p | q_1, p \neq 2$,
\begin{align}
\label{Dgenl0}	
\begin{split}
  & \sum_{l=0}^\infty \frac{ \psi(p^l)G\lz \chi_{p^l}, q_1 \pz r(p^l, t)}{p^{lw}} = 1-\frac{ \psi(p^2)}{p^{2w-1}}\lz \lambda_f(p^2)-\frac{\lambda^2_f(p)}{p^{t}}+\frac{1}{p^{2t}}\pz =  1+O(p^{-2w-1}+p^{-2w-2t-1}).
\end{split}
\end{align}

  We thus deduce from \eqref{Dexp}, \eqref{Dk1est} and \eqref{Dgenl0}	 that for $p | q_1$, $p \neq 2$, $\Re(s)>1/2$, $\Re(w)>1$, $\Re(w+t)>1$,
\begin{align}
\label{Dgenexp1}	
\begin{split}
   & D_{1,p}(s, w,t; q_1, \psi,\psi',f)
=  1+O \Big (p^{-2w-1}+p^{-2w-2t-1}+ p^{-2s}+p^{-2s-w-t+1}\Big )  .
\end{split}
\end{align}

  We conclude from \eqref{D1Eulerprod}, \eqref{Dexp}, \eqref{Dgenexp} and \eqref{Dgenexp1} that for $\Re(s)>1+\varepsilon$,  $\Re(w)>1+\varepsilon$ and $\Re(w+t)>1+\varepsilon$,
\begin{align*}
\begin{split}
	D_{1}(s, w,t;q_1, \psi,\psi',f) \ll	& q_1^{\varepsilon}\Big |\frac {L^{(2q_1)}\lz w-\tfrac{1}{2}, f \otimes \chi^{(q_1)}\psi\pz}{L^{(2q_1)}(2w-1, \sym f)\zeta^{(2q_1)}(2w-1)L^{(2q_1)}\lz t+w-1/2,f \otimes \chi^{(q_1)}\psi\pz} \Big | \\
\ll & q_1^{\varepsilon}\Big |\frac {L\lz w-\tfrac{1}{2}, f \otimes \chi^{(q_1)}\psi\pz}{L(2w-1, \sym f)\zeta(2w-1)L\lz t+w-1/2,f \otimes \chi^{(q_1)}\psi\pz} \Big |
\ll |q_1w(w+t)|^{\varepsilon},
\end{split}
\end{align*}
  where the last estimation above follows \eqref{PgLest1}. This implies \eqref{Diest} and hence completes the proof of the lemma.
\end{proof}
	
  The above lemma now allows us to extend $C(s,w,z;f)$ to the region
\begin{equation*}
		\{(s,w,z):\ \Re(s)>1, \ \Re(w)>1,\ \Re(z)>1\}.
\end{equation*}
	Using \eqref{S3}, \eqref{A1A2} and the above, we can extend $(s-1)(w-1)A(s,w,z;f)$ to the region
\begin{align*}
		S_4=& \{(s,w,z):\Re(2z)>1, \ \Re(s+2z)>1, \ \Re(s+2w)>1,\ \Re(w+z)>1, \ \Re(4w)>1, \ \Re(s+w+z)>1, \\
& \hspace{1in} \ \Re(s+w)>1,\ \Re(s+z)>1, \ \Re(1-s)>1\}.
\end{align*}
   Note that the condition $\Re(1-s)>1$ is equivalent to $\Re(s)<0$ so that the conditions $\Re(s+w)>1, \Re(s+z)>1$ is the same as $\Re(w)>1$, $\Re(z)>1$.   It follows that the rest of the conditions given in the definition of $S_4$ are superseded by the above three conditions.  Thus
\begin{equation*}
		S_4=\{(s,w,z):\ \Re(s)<0, \ \Re(s+w)>1,\ \Re(s+z)>1\}.
\end{equation*}

   We further note that the region $S_2$ contains the subset given by
\begin{equation*}
	\{(s,w,z): \Re(z)>1, \ \Re(s)>1, \ \Re(s+2w)>2 \}.
\end{equation*}

  As the region $S_4$ contains points $(s,w,z)$ such that
\begin{equation*}
	\{ \Re(s)<0, \ 1< \Re(s+w)<\Re(z), \ \Re(s)>1-\Re(z) \},
\end{equation*}
  it is then readily seen that the convex hull of the above regions contains $S_5 \cap \{ (s,w,z): \Re(z)>1\}$, where
\begin{align*}
		S_5=\{(s,w,z):\ &\Re(s+2w)>2,\ \Re(s+2z)>2, \ \Re(s+z)>1, \ \Re(s+w)>1, \ \Re(w)> \tfrac{1}{4}, \ \Re(z)>\tfrac{1}{2} \}.
\end{align*}

 On the other hand, when $1/2 < \Re(z)< 1$, we note that the points in $S_2$ are certainly contained in the convex hull of $S_2$ and $S_4$ and one checks $S_2 \cap \{ (s,w,z): \Re(s)>1\}=S_5 \cap\{ (s,w,z): \Re(s)>1\}$. We may thus focus on the case $\Re(s)<1$.  In this case, we note that the points $(1,1/2, 1/2)$, $(1,1/2,1)$ are in the closure of $S_2$ and the point $(0,1,1)$ is in the closure of $S_4$. These three points determine a triangular region which can be regarded as the base of the region $\mathcal S$ enclosed by the four planes: $\Re(s+2w)=2$, $\Re(s)=1$, $\Re(z)=1$, $\Re(s+2z)=2$. It follows that the points in this triangular region are all in the convex hulls of $S_2$ and $S_4$. Further note that the points on the boundary $\mathcal S \cap \{\Re(z)=1\}$ of $\mathcal S$ are all in the convex hull of $S_2$ and $S_4$ since they can be identified with the set $\{(s, w, z): \Re(z)=1, 0 \leq \Re(s) \leq 1, \Re(s+2w) \geq  2\}$ and hence is contained in the convex hull of $S_5 \cap \{ (s,w,z): \Re(z)>1\}$.  We then deduce that the entire region $\mathcal S$ lies in the convex hull of $S_0$ and $S_1$.  We next note that when $1/2 < \Re(z)< 1$, the condition $\Re(s+2z)>2$ implies that $\Re(s)>0$ so that one also has $\Re(s+z)>1$ and that the condition $\Re(s+2w)>2$ implies $ \Re(s+w)>1$. It follows from these discussions we see that the intersection of the convex hull of $S_2$ and $S_4$ thus contains $S_5$. \newline

We apply Theorem \ref{Bochner} again to conclude that $(s-1)(w-1)A(s,w,z;f)$ converges absolutely in the region $S_5$.

\subsection{Bounding $A(s,w, z;f)$ in vertical strips}
\label{Section bound in vertical strips}
	
In order to prove Theorem \ref{Theorem for all characters}, we also need to estimate $|A(s,w,z;f)|$ in vertical strips. \newline

We set for any fixed $0<\delta <1/1000$ and the previously defined regions $S_j$,
\begin{equation*}
		\widetilde S_j=S_{j,\delta}\cap\{(s,w,z):\Re(s)>-5/2,\ \Re(w)>1/2-\delta\},
\end{equation*}
	where $S_{j,\delta}= \{ (s,w,z)+\delta (1,1,1) : (s,w,z) \in S_j \} $.  Set
\begin{equation*}
		p(s,w)=(s-1)(w-1),
\end{equation*}
so that $p(s,w)A(s,w,z;f)$ is an analytic function in the regions under our consideration.  We also write $\tilde p(s,w)=1+|p(s,w)|$. \newline

We consider the bound for $A(s,w, z;f)$ given in \eqref{Aboundinitial} and apply \eqref{Lfchibound1} to deduce that, under GRH, in $\widetilde S_0$,
\begin{align*}
\begin{split}
      |p(s,w)A(s,w,z;f)| \ll \tilde p(s,w)|wz|^{\varepsilon}(1+|w|)^{\max \{1-2\Re(w), 0 \}+\varepsilon}.
\end{split}
\end{align*}

   Similarly, using the estimates \eqref{Lchibound1} in \eqref{Sum A(s,w,z) over n} renders that, under GRH, in the region $\widetilde S_1$,
\begin{align*}
	|p(s,w)A(s,w,z;f)|\ll \tilde p(s,w)(1+|s|)^{\max \{1/2-\Re(s), 0\}+\varepsilon}.
\end{align*}
   We then deduce from the above and Proposition \ref{Extending inequalities} that in the convex hull $\widetilde S_2$ of $\widetilde S_0$ and $\widetilde S_1$, we have under GRH,
\begin{equation}
\label{AboundS2}
		|p(s,w)A(s,w,z;f)|\ll \tilde p(s,w) |wz|^{\varepsilon}(1+|w|)^{\max \{1-2\Re(w),0 \}+\varepsilon}(1+|s|)^{3+\varepsilon}.
\end{equation}

Moreover, by \eqref{residuesgensimplified} and the estimations given in \eqref{Lchibound1} for $\zeta(s)$ (corresponding to the case $\psi=\psi_0$ being the primitive principal character) and in \eqref{Lsymbound} for $L^{(2)}(2w, \sym f)$ that in the region $\widetilde S_3$, under GRH,
\begin{align}
\label{A1bound}
		|A_1(s,w,z;f)| \ll |w|^{\varepsilon}(1+|s|)^{\max \{(1-\Re(s))/2, 1/2-\Re(s), 0\}+\varepsilon}.
\end{align}

   Also, by \eqref{Cexp}--\eqref{Didef} and Lemma \ref{Estimate For D(w,t)},
\begin{equation}
\label{Csbound}
		|C(s,w,z;f)|\ll |wz|^{\varepsilon}
\end{equation}
   in the region
\begin{equation*}
		\{(s,w,z):\Re(w)>1+\varepsilon,\ \Re(z)>1+\varepsilon,\ \Re(s) >1+\varepsilon\}.
\end{equation*}

  Now, applying \eqref{A1A2}, the functional equation \eqref{Functional equation in s}, the bounds given in \eqref{A1bound}, \eqref{Csbound}, together with \eqref{Stirlingratio}, we obtain that in the region $\widetilde S_4$,
\begin{equation}
\label{AboundS3}
		|p(s,w)A(s,w,z;f)|\ll \tilde p(s,w) |wz|^{\varepsilon}(1+|w|)^{\max \{3(1/2-\Re (w)), 0\} +\varepsilon}(1+|s|)^{3+\varepsilon}.
\end{equation}

  We now conclude from \eqref{AboundS2}, \eqref{AboundS3} and Proposition \ref{Extending inequalities} that in the convex hull $\widetilde S_5$ of $\widetilde S_2$ and $\widetilde S_4$,
\begin{equation}
\label{AboundS4}
		|p(s,w)A(s,w,z;f)|\ll \tilde p(s,w)|wz|^{\varepsilon}(1+|w|)^{\max \{3(1/2-\Re (w)), 0\} +\varepsilon}(1+|s|)^{3+\varepsilon}.
\end{equation}
	
\subsection{Completion of proof}

The Mellin inversion yields that
\begin{equation}
\label{Integral for all characters}
		\sum_{\substack{(n,2)=1}}\frac{L^{(2)}(\tfrac{1}{2}+\alpha, f \otimes \chi_{n})}{L^{(2)}(\tfrac{1}{2}+\beta, f \otimes \chi_{n})}w \bfrac {n}X=\frac1{2\pi i}\int\limits_{(2)}A\lz s,\tfrac12+\alpha,\tfrac12+\beta; f \pz X^s\widehat w(s) \dif s,
\end{equation}
  where $A(s, w, z;f)$ is defined in \eqref{Aswzexp} and where we recall that $\widehat{w}$ is the Mellin transform of $w$ defined by
\begin{align*}
     \widehat{w}(s) =\int\limits^{\infty}_0w(t)t^s\frac {\dif t}{t}.
\end{align*}
Now repeated integration by parts gives that for any integer $E \geq 0$,
\begin{align}
\label{whatbound}
 \widehat w(s)  \ll  \frac{1}{(1+|s|)^{E}}.
\end{align}

    We evaluate the integral in \eqref{Integral for all characters} by shifting the line of integration  to $\Re(s)=N(\alpha,\beta)+\varepsilon$,  where $N(\alpha,\beta)$ is given in \eqref{Nab}.  Applying \eqref{AboundS4} and \eqref{whatbound} gives that the integral on the new line can be absorbed into the $O$-term in \eqref{Asymptotic for ratios of all characters}.  We encounter a simple pole at $s=1$ in the process whose residue is given in \eqref{Residue at s=1}. This yields the main terms in \eqref{Asymptotic for ratios of all characters} and completes the proof of Theorem \ref{Theorem for all characters}.

\vspace*{.5cm}

\noindent{\bf Acknowledgments.}   P. G. is supported in part by NSFC grant 11871082 and L. Z. by the FRG Grant PS43707 at the University of New South Wales.

\bibliography{biblio}
\bibliographystyle{amsxport}

\end{document}